\documentclass{amsart}
\usepackage{amssymb, amsmath, amsthm, graphics, comment, xspace, enumerate}
\newtheorem{thm}{Theorem}[section]
\newtheorem{cor}[thm]{Corollary}
\newtheorem{lem}[thm]{Lemma}
\newtheorem{prop}[thm]{Proposition}
\theoremstyle{definition}
\newtheorem{defn}[thm]{Definition}
\newtheorem{example}[thm]{Example}
\theoremstyle{remark}

\numberwithin{equation}{section}
\begin{document}
\title[$c$-Almost periodic type functions and applications]{$c$-Almost periodic type functions and applications}

\author{M. T. Khalladi}
\address{Department of Mathematics and Computer Sciences, 
University of Adrar, Adrar, Algeria}
\email{ktaha2007@yahoo.fr}

\author{M. Kosti\' c}
\address{Faculty of Technical Sciences,
University of Novi Sad,
Trg D. Obradovi\' ca 6, 21125 Novi Sad, Serbia}
\email{marco.s@verat.net}

\author{M. Pinto}
\address{Departamento de
Matem\'aticas, Facultad de Ciencias, Universidad de Chile, Santiago de Chile, Chile}
\email{pintoj.uchile@gmail.com}

\author{A. Rahmani}
\address{Laboratory of Mathematics, Modeling and Applications (LaMMA), University of Adrar, Adrar, Algeria}
\email{vuralrahmani@gmail.com}

\author{D. Velinov}
\address{Faculty of Civil Engineering, Ss. Cyril and Methodius University, Skopje,
Partizanski Odredi
24, P.O. box 560, 1000 Skopje, North Macedonia}
\email{velinovd@gf.ukim.edu.mk}

{\renewcommand{\thefootnote}{} \footnote{
\\ \text{  }  \ \    {\it Key words and phrases.} $c$-almost periodic functions, $c$-uniformly recurrent functions, semi-$c$-periodic functions, convolution products, abstract fractional semilinear inclusions.
\\  \text{  }  \ \ Mohammed Taha Khalladi, Marko Kosti\' c, Manuel Pinto, Abdelkader Rahmani, Daniel Velinov.\\
Marko Kosti\' c and Daniel Velinov are partially supported by grant 174024 of Ministry
of Science and Technological Development, Republic of Serbia and Bilateral project between MANU and SANU.
Manuel Pinto is partially supported by Fondecyt 1170466.}}

\begin{abstract}
In this paper, we introduce several various classes of $c$-almost periodic type functions and their Stepanov generalizations, where $c\in {\mathbb C}$ and $|c|=1.$
We also consider the corresponding classes of $c$-almost periodic type functions 
depending on two variables and prove several related composition principles.
Plenty of illustrative examples and applications are presented.
\end{abstract}
\maketitle

\section{Introduction and preliminaries}

The notion of almost periodicity was studied by H. Bohr around 1925 and later generalized by many others. The interested reader may consult 
the monographs by Besicovitch \cite{besik}, Bezandry and Diagana \cite{Beza}, Corduneanu \cite{corduneanu}-\cite{c-cord}, Diagana \cite{diagana}, Fink \cite{fink}, Gu\' er\' ekata \cite{gaston},
Kosti\' c \cite{nova-mono} and Zaidman \cite{30} for the basic introduction to the theory of almost periodic functions. 
Almost periodic functions and almost automorphic functions play a significant role in the qualitative theory of differential equations, physics, mathematical biology, control theory 
and technical sciences.

The class of $(\omega,c)$-periodic functions and various generalizations have recently been introduced and investigated by Alvarez, G\'omez, Pinto \cite{alvarez1} and
Alvarez, Castillo, Pinto \cite{alvarez2}-\cite{alvarez3} (see also Pinto \cite{pintom}, where the notion of $(\omega,c)$-periodicity has been analyzed for the first time). In \cite{abdoje}-\cite{facta}, we have recently considered various generalizations of $(\omega,c)$-periodic functions. 
Besides the notion depending on two  parameters $\omega$ and $c$, it is meaningful to consider the notion depending only on the parameter $c.$ 
The main aim of this paper is to introduce and analyze the classes of
$c$-almost periodic functions, $c$-uniformly recurrent functions, semi-$c$-periodic functions 
and their Stepanov generalizations (see \cite{chelj} for further information concerning 
semi-Bloch $k$-periodic functions ($k\in {\mathbb{R}}$) and
semi-anti-periodic functions; the class of uniformly recurrent functions has recently been analyzed in \cite{densities}), where $c\in {\mathbb C}$ and $|c|=1.$ We also introduce and investigate the corresponding classes of $c$-almost periodic type functions 
depending on two variables; several composition principles for $c$-almost periodic type functions are established in this direction.
We provide some illustrative examples and applications to
the abstract fractional semilinear integro-differential inclusions.

Before briefly explaining the organization of paper and notation used, the authors would like to thank Professor Toka Diagana for his invitation to submit this article to the special issue of Nonautonomous Dynamical Systems dedicated 
to the memory of Professor Constantin Corduneanu.

By $(E,\|\cdot\|)$ we denote a complex Banach space; $I$ denotes the interval ${\mathbb R}$ or $[0,\infty).$ If $X$ is also a complex Banach space, then 
$L(E,X)$ stands for the space of all continuous linear mappings from $E$ into
$X;$ $L(E)\equiv L(E,E).$ 

In Subsection \ref{madaj} we collect the basic definitions and results about almost periodic type functions and 
their Stepanov generalizations. The notion of $c$-almost periodicity and the notion of $c$-uniform recurrence
are introduced in Definition \ref{drmniga} and Definition \ref{onakao}, respectively (if $c=1,$ then we recover the usual notions of almost periodicity and uniform recurrence). After that, in Definition \ref{semi-ceanti} and Proposition \ref{netodin}, we introduce the notion of semi-$c$-periodicity and prove some necessary and sufficient conditions for a continuous function $f: I \rightarrow E$ to be semi-$c$-periodic.

Proposition \ref{dinko} is important in our analysis because it states that there does not exist a $c$-uniformly recurrent function $f: I \rightarrow E$ if $|c|\neq 1$ (see also 
Proposition \ref{durak-durak} and Proposition \ref{dusanka54321} for some expected results on the introduced classes of functions). Further on, in Proposition \ref{uspevanje}, 
we show that for any $c$-almost periodic function ($c$-uniformly recurrent function, semi-$c$-periodic function)
$f : I \rightarrow E$ and for any positive integer $l\in {\mathbb N}$ the function
$f(\cdot)$ is $c^{l}$-almost periodic ($c^{l}$-uniformly recurrent, semi-$c^{l}$-periodic); after that, in Corollary \ref{uspeva}, we clarify the basic consequences in case that $\arg(c)\in \pi \cdot {\mathbb Q}.$ 
In Proposition \ref{matsah}, we analyze the situation in which $\arg(c)\notin \pi \cdot {\mathbb Q};$ if this is the case, then  
we prove that the $c$-almost periodicity of function $f(\cdot)$ implies the $c'$-almost periodicity of function $f(\cdot)$ for all $c'\in S_{1},$ where $S_{1}:=\{ z\in {\mathbb C} \, ; \, |z|=1\}$. Furthermore,  
if the function $f(\cdot)$ is bounded and $c$-uniformly recurrent, then $f(\cdot)$ is $c'$-uniformly recurrent for all $c'\in S_{1}$ (to sum up, any $c$-almost periodic function is always almost periodic and any bounded $c$-uniformly recurrent function is always uniformly recurrent; the converse statement is false in general); see also 
Proposition \ref{matsahari}.

The main structural properties of introduced classes of functions are stated in Theorem \ref{krewfar}.
After that, in Theorem \ref{petruciani}, we clarify that 
a bounded continuous function
$f : I  \rightarrow E$ is semi-$c$-periodic if and only if there exists a sequence $(f_{n})$ of bounded continuous $c$-periodic functions 
which uniformly converges to $f(\cdot).$ If the function $f(\cdot)$ is real-valued, $c$-uniformly recurrent (semi-$c$-periodic) and $f\neq  0,$ then $c=\pm 1$ and moreover, if $f(t)\geq 0$ for all $t\in I,$ then $c=1;$ 
see
Proposition \ref{gruskamiroslavglisiclord}. Furthermore, if 
$f : I\rightarrow E$ is $c$-uniformly recurrent (semi-$c$-periodic) and $f\neq  0,$ then $f(\cdot)$ cannot vanish at infinity; see Proposition \ref{balija}. 

An interesting extension of \cite[Theorem 2.3]{krag} is proved in Theorem \ref{deckoracunao}, which is one of the main results of the second section. The $c$-almost periodic extensions (semi-$c$-periodic extensions) of functions from the nonnegative real axis to the whole real axis are analyzed in Proposition \ref{batty}.
The notion of asymptotical $c$-almost periodicity (asymptotical $c$-uniform recurrence, asymptotical semi-$c$-periodicity) is introduced in Definition \ref{asymp-cea}, while the notion of 
corresponding Stepanov  classes is introduced in Definition \ref{stepa-anti-c}. 

Without going into full details, 
let us only say that the composition theorems for $c$-almost periodic type functions are analyzed in Subsection \ref{poredfgha}, while 
the invariance of $c$-almost type periodicity under the actions of convolution products is analyzed in Subsection \ref{stepamojalo} (the structural results in these subsections are given without proofs, which can be deduced similarly as in our previous research studies; it is also worth noting that, in Section 2, we present 
numerous illustrative examples and comments about the problems considered).
The final section of paper is reserved for applications of our abstract theoretical results. 

 By
$L^{p}_{loc}(I :E)$, $C(I : E),$ $C_{b}(I : E)$ and $C_{0}(I : E)$ we denote the vector spaces consisting of all $p$-locally integrable functions $f  : I \rightarrow E$, all continuous functions $f  : I \rightarrow E,$ all bounded continuous functions $f  : I \rightarrow E$ and all continuous functions $f  : I \rightarrow E$
satisfying that $\lim_{|t|\rightarrow +\infty}\|f(t)\|=0,$ respectively ($1\leq p <\infty$). As is well known, $C_{0}(I : E)$ is a Banach space equipped with the sup-norm, denoted henceforth by $\|\cdot\|_{\infty}.$ If $f : {\mathbb R} \rightarrow E$, then we define $\check{f} : {\mathbb R} \rightarrow E$ by $\check{f}(x):=f(-x),$ $x\in {\mathbb R}.$

We will use the following auxiliary result, whose proof follows from the argumentation used in the proof that every orbit under an irrational rotation is dense in $S_{1}$
(see e.g. the solution given by C. Blatter in \cite{blater}):

\begin{lem}\label{rad98765}
Suppose that $c=e^{i\pi \varphi},$ where $\varphi \in (-\pi,\pi]\setminus \{0\}$ is not rational. Then for each $c'\in S_{1}$ there exists a 
strictly increasing sequence $(l_{k})$ of positive integers such that $\sup_{k\in {\mathbb N}}(l_{k+1}-l_{k})<\infty$ and $|c^{l_{k}}-c'|<\epsilon.$
\end{lem}

%\begin{proof}
%liter
%\end{proof}

\subsection{Almost periodic type functions and generalizations}\label{madaj}

Given $\epsilon>0,$ we call $\tau>0$ an $\epsilon$-period for $f(\cdot)$ if
\begin{align*}
\| f(t+\tau)-f(t) \| \leq \epsilon,\quad t\in I.
\end{align*}
The set constituted of all $\epsilon$-periods for $f(\cdot)$ is denoted by $\vartheta(f,\epsilon).$ It is said that $f(\cdot)$ is almost periodic
if for each $\epsilon>0$ the set $\vartheta(f,\epsilon)$ is relatively dense in $[0,\infty),$ which means that
there exists $l>0$ such that any subinterval of $[0,\infty)$ of length $l$ meets $\vartheta(f,\epsilon)$. The vector space consisting of all almost periodic functions is denoted by $AP(I :E).$
This space contains the space consisting of all continuous periodic functions $f : I\rightarrow E.$

Let $f \in AP(I : E).$ Then the Bohr-Fourier coefficient
$$
P_{r}(f) := \lim_{t\rightarrow \infty}\frac{1}{t}\int^{t}_{0}e^{-irs}f(s)\, ds
$$ 
exists for all $r\in {\mathbb R};$ 
furthermore, if $P_{r}(f) = 0$ for all $r \in {\mathbb R},$ then $f(t) = 0$ for all $t \in {\mathbb R},$
and $\sigma(f):=\{r\in {\mathbb R} :  P_{r}(f) \neq 0\}$ is at most countable.
The function $f : I \rightarrow E$ is said to be asymptotically almost periodic if and only if
there exist an almost periodic function $h : I  \rightarrow E$ and a function $\phi \in  C_{0}(I  : E)$
such that $f(t) = h(t)+\phi (t)$ for all $t\in I$. This is equivalent to saying that, for every $\epsilon >0,$ we can find numbers $ l > 0$ and $M >0$ such that every subinterval of $I$ of
length $l$ contains, at least, one number $\tau$ such that $\|f(t+\tau)-f(t)\| \leq \epsilon$ provided $|t|,\ |t+\tau|\geq M.$ 

For any almost periodic function $f : I \rightarrow E$, the spectral synthesis states that
\begin{align}\label{zaq54321}
f\in \overline{span\{ e^{i\mu \cdot } x : \mu \in \sigma(f),\ x\in R(f)\}},
\end{align}
where the closure is taken in the space $C_{b}(I: E).$ By $AP_{\Lambda}(I:E),$ where $ \Lambda$ is a non-empty subset of ${\mathbb R},$ we denote the vector subspace of $AP(I:E)$ consisting of all functions $f\in AP(I:E)$ for which the inclusion $\sigma(f)\subseteq \Lambda$ holds. We have that
$AP_{\Lambda}(I:E)$ is a closed subspace of $AP(I:E)$ and therefore Banach space itself.

For the sequel, we need some preliminary results from the pioneering paper \cite{Ba} by Bart and Goldberg, who introduced the notion of an almost periodic strongly continuous semigroup there.
The translation semigroup $(W(t))_{t\geq 0}$ on $AP([0,\infty) : E),$ given by $[W(t)f](s):=f(t+s),$ $t\geq 0,$ $s\geq 0,$ $f\in AP([0,\infty) : E)$ is consisted solely of surjective isometries $W(t)$ ($t\geq 0$) and can be extended to a $C_{0}$-group $(W(t))_{t\in {\mathbb R}}$ of isometries on $AP([0,\infty) : E),$ where $W(-t):=W(t)^{-1}$ for $t>0.$ Furthermore, the mapping ${\mathbf E} : AP([0,\infty) : E) \rightarrow AP({\mathbb R} : E),$ defined by
$$
[{\mathbf E}f](t):=[W(t)f](0),\quad t\in {\mathbb R},\ f\in AP([0,\infty) : E),
$$
is a linear surjective isometry and ${\mathbf E}f(\cdot)$ is a unique continuous almost periodic extension of a function $f(\cdot)$ from $AP([0,\infty) : E)$ to the whole real line. 

Following Haraux and Souplet \cite{haraux}, we say that a continuous function $f(\cdot)$ is uniformly recurrent if and only if there exists a strictly increasing sequence $(\alpha_{n})$ of positive real numbers such that $\lim_{n\rightarrow +\infty}\alpha_{n}=+\infty$ and
\begin{align*}
\lim_{n\rightarrow \infty}\sup_{t\in {\mathbb R}}\bigl\|f(t +\alpha_{n})-f(t)\bigr\|=0.
\end{align*}
It is well known that any almost periodic function is uniformly recurrent, while the converse statement is not true in general. For more details about uniformly recurrent functions, we refer the reader to \cite{densities}.

Let $1\leq p <\infty.$ We continue by recalling that a function $f\in L^{p}_{loc}(I :E)$ is said to be Stepanov $p$-bounded if and only if
$$
\|f\|_{S^{p}}:=\sup_{t\in I}\Biggl( \int^{t+1}_{t}\|f(s)\|^{p}\, ds\Biggr)^{1/p}<\infty.
$$
Equipped with the above norm, the space $L_{S}^{p}(I:E)$ consisted of all Stepanov $p$-bounded functions is a Banach space.
A function $f\in L_{S}^{p}(I:E)$ is said to be Stepanov $p$-almost periodic if and only if the function
$
\hat{f} : I \rightarrow L^{p}([0,1] :E),
$ defined by
\begin{align}\label{zad}
\hat{f}(t)(s):=f(t+s),\quad t\in I,\ s\in [0,1],
\end{align}
is almost periodic. Furthermore, we say that a function $f\in L_{S}^{p}(I :E)$ is asymptotically Stepanov $p$-almost periodic if and only if there exist a Stepanov $p$-almost periodic
function $g\in  L_{S}^{p}(I : E)$ and a function $q\in  L_{S}^{p}(I  : E)$ such that $f(t)=g(t)+q(t),$ $t\in I$ and $\hat{q}\in C_{0}(I  : L^{p}([0,1]: E)).$

We also need the following definition from \cite{densities}.

\begin{defn}\label{ret}
Let $1\leq p<\infty.$
\begin{itemize}
\item[(i)] 
A function $f\in L_{loc}^{p}(I:E)$ is said to be Stepanov $p$-uniformly recurrent  if and only if the function
$
\hat{f} : I \rightarrow L^{p}([0,1] :E),
$ defined by \eqref{zad},
is uniformly recurrent.
\item[(ii)] 
A function $f\in L_{loc}^{p}(I  :E)$ is said to be asymptotically Stepanov $p$-uniformly recurrent
if and only if there exist a Stepanov $p$-uniformly recurrent
function $h(\cdot)$ and a function $q\in  L_{S}^{p}(I : E)$ such that $f(t)=h(t)+q(t),$ $t\in I$ and $\hat{q}\in C_{0}(I  : L^{p}([0,1]: E)).$
\end{itemize}
\end{defn}

\section[$c$-Almost periodic type functions]{$c$-Almost periodic type functions}\label{c-almost}

With the exception of Proposition \ref{dinko} and the paragraph preceding it, in this paper we will always assume that $c\in {\mathbb C}$ and $|c|=1.$
Let $f: I \rightarrow E$ be a continuous function and let a number $\epsilon>0$ be given. We call a number $\tau >0$ an $(\epsilon,c)$-period for $f(\cdot)$ if $\|f(t+\tau)-cf(t)\|\leq\epsilon$ for all $t\in I$. 
By $\vartheta_{c}(f,\epsilon)$ we denote the set consisting of all $(\epsilon ,c)$-periods for $f(\cdot)$. 

We are concerned with the following notion:

\begin{defn}\label{drmniga}
It is said that $f(\cdot)$ is $c$-almost periodic if and only if for each $\epsilon>0$ the set $\vartheta_{c}(f,\epsilon)$ is relatively dense in $[0,\infty).$
The space consisting of all $c$-almost periodic functions from the interval $I$ into $E$ will be denoted by $AP_{c}(I:E)$.
\end{defn}

If $c=-1,$ then we also say that the function $f(\cdot)$ is almost anti-periodic. The space of almost anti-periodic functions has recently been analyzed in \cite{krag}.

In general case, it is very simple to prove that the following holds (see e.g., the proof of \cite[Theorem 4$^{\circ}$, p. 2]{besik}):

\begin{prop}\label{ogranicenost}
Suppose that $f : I \rightarrow E$ is $c$-almost periodic. Then $f(\cdot)$ is bounded.
\end{prop}

The following generalization of $c$-almost periodicity is meaningful, as well:

\begin{defn}\label{onakao}
Let $c\in {\mathbb C} \setminus \{0\}.$ Then a continuous function $f : I \rightarrow E$ is said to be $c$-uniformly recurrent if and only if there exists a strictly increasing sequence $(\alpha_{n})$ of positive real numbers such that $\lim_{n\rightarrow +\infty}\alpha_{n}=+\infty$ and
\begin{align}\label{rastvor}
\lim_{n\rightarrow +\infty}\bigl\| f(\cdot +\alpha_{n})-cf(\cdot) \bigr\|_{\infty}=0.
\end{align}
If $c=-1,$ then we also say that the function $f(\cdot)$ is uniformly anti-recurrent.
The space consisting of all $c$-uniformly recurrent functions from the interval $I$ into $E$ will be denoted by $UR_{c}(I:E)$.
\end{defn}

Define now ${\mathbb S}:={\mathbb N}$ if $I=[0,\infty),$ and
${\mathbb S}:={\mathbb Z}$ if $I={\mathbb R}.$ We will also consider the following notion:

\begin{defn}\label{semi-ceanti}
Let $f\in C(I:E).$
It is said that $f(\cdot )$ is semi-$c$-periodic if and only if
\begin{equation*}
\forall \varepsilon >0\quad \exists p> 0\quad \forall m\in {\mathbb{S }}%
\quad \forall x\in I \quad \bigl\| f(x+mp)-c^{m}f(x)\bigr\| \leq
\varepsilon .
\end{equation*}
The space of all semi-$c$-periodic functions will be denoted by 
$
\mathcal{SAP}_{c}(I:E).$ 
\end{defn}

Suppose that $I={\mathbb R},$ $f\in C({\mathbb R}:E),$ $p>0$ and $m\in {\mathbb N}.$ Then we have
\begin{align*}
\sup_{x\in {\mathbb R}}&\bigl\| f(x+mp)-c^{m}f(x) \bigr\|
=\sup_{x\in {\mathbb R}}\bigl\| f(x)-c^{m}f(x-mp) \bigr\|
\\& =\sup_{x\in {\mathbb R}}\bigl\| c^{m}\bigl[c^{-m}f(x)-f(x-mp)\bigr] \bigr\|
=|c|^{m}\sup_{x\in {\mathbb R}}\bigl\| f(x-mp)-c^{-m}f(x) \bigr\|
\\& = \sup_{x\in {\mathbb R}}\bigl\| f(x-mp)-c^{-m}f(x) \bigr\| \in [0,\infty].
\end{align*}

Therefore, we have the following:

\begin{prop}\label{netodin}
Suppose that 
$f\in C(I:E).$ Then $f(\cdot )$ is semi-$c$-periodic if and only if 
\begin{equation*}
\forall \varepsilon >0\quad \exists p> 0\quad \forall m\in {\mathbb{N }}%
\quad \forall x\in I \quad \bigl\| f(x+mp)-c^{m}f(x)\bigr\| \leq
\varepsilon .
\end{equation*}
Furthermore, if $I={\mathbb R},$ then the above is also equivalent with 
\begin{equation*}
\forall \varepsilon >0\quad \exists p> 0\quad \forall m\in -{\mathbb{N }}%
\quad \forall x\in I \quad \bigl\| f(x+mp)-c^{m}f(x)\bigr\| \leq
\varepsilon .
\end{equation*}
\end{prop}

It can be very simply shown that any 
semi-$c$-periodic function is bounded.
Keeping in mind Proposition \ref{netodin} and this observation, we may conclude that the notion introduced in Definition \ref{semi-ceanti} is equivalent and extends the notion of semi-periodicity for case $c=1,$
introduced by Andres and Pennequin in \cite{andres1}, and the notion of semi-anti-periodicity for case $c=-1,$ introduced by Chaouchi et al in \cite{chelj}. 

The notion introduced in Definition \ref{semi-ceanti} can be considered with general complex number
$c\in {\mathbb C}\setminus \{0\},$ but the situation is much more complicated in this case (cf. \cite{novipinto} for more details). The same holds with the notion introduced in Definition \ref{drmniga} and Definition \ref{onakao},
but then we have the following result: 

\begin{prop}\label{dinko}
Suppose that  $f\in UR_{c}(I:E)$ and $c\in {\mathbb C} \setminus \{0\}$ satisfies $|c|\neq 1.$ Then $f\equiv 0.$ 
\end{prop}

\begin{proof}
Without loss of generality, we may assume that $I=[0,\infty).$ 
Suppose to the contrary that there exists $t_{0}\geq 0$ such that $f(t_{0})\neq 0.$
Inductively, \eqref{rastvor} implies
\begin{align}\label{ironija}
|c|^{k}m-\frac{|c|^{k}-1}{n(|c|-1)}\leq \|f(t)\|\leq |c|^{k}M-\frac{|c|^{k}-1}{n(|c|-1)},\quad k\in {\mathbb N},\ t\in \bigl[k\alpha_{n},(k+1)\alpha_{n} \bigr].
\end{align}
Consider now case $|c|<1.$ Let $0<\epsilon<c\|f(t_{0})\|.$ Then \eqref{ironija} yields that there exist integers $k_{0}\in {\mathbb N}$ and $n\in {\mathbb N}$ such that 
for each $k\in {\mathbb N}$ with $k\geq k_{0}$ we have
$\|f(t)\| \leq \epsilon/2,$ $t\in [k\alpha_{n},(k+1)\alpha_{n}].$ Then the contradiction is obvious because for each $m\in {\mathbb N}$ with $m>n$ there exists $k\in {\mathbb N}$ such that $t_{0}+\alpha_{m}\in [k\alpha_{n},(k+1)\alpha_{n} ]$ and therefore
$\|f(t_{0}+\alpha_{m})\|\geq c\| f(t_{0})\|-(1/m)\rightarrow  c\| f(t_{0})\|>\epsilon,$ $m\rightarrow +\infty.$
Consider now case $|c|>1;$ let $n\in {\mathbb N}$ be such that $\|f(t_{0})\|>1/(n(|c|-1))$ and $M:=\max_{t\in [0,2\alpha_{n}]}\|f(t)\|>0.$
Then for each $m\in {\mathbb N}$ with $m>n$ there exists $k\in {\mathbb N}$ such that $\alpha_{m}\in [(k-1)\alpha_{n},k\alpha_{n}]$ and therefore
$\|f(t+\alpha_{m})\|\leq 1+|c|M,$ $t\in [0,2\alpha_{n}].$ On the other hand, we obtain inductively from \eqref{rastvor}
that
$$
\bigl\| f(t_{0}+k\alpha_{n})\bigr\|\geq |c|^{k}\Biggl[ \bigl\|f(t_{0})\bigr\|-\frac{1}{n(|c|-1)} \Biggr]+\frac{1}{n(|c|-1)}\rightarrow +\infty \ \ \mbox{ as } k\in {\mathbb N},
$$
which immediately yields a contradiction.
\end{proof}

Using the same arguments as in the proof of \cite[Lemma 3.4]{abdoje}, we can clarify the following: 

\begin{prop}\label{durak-durak} 
Suppose that $I={\mathbb R}$ and $f : {\mathbb R} \rightarrow E.$ Then the function $f(\cdot)$ is $c$-almost periodic ($c$-uniformly recurrent, semi-$c$-periodic)
if and only if the function $\check{f}(\cdot)$ is $1/c$-almost periodic  ($1/c$-uniformly recurrent, semi-$1/c$-periodic).
\end{prop}

Since for each numbers $t,\ \tau \in I$ and $m\in {\mathbb N}
$ we have 
\begin{align*}
\Bigl| \bigl\|f(t+\tau)\bigr\|-\|f(t)\|\Bigr|=\Bigl| \bigl\|f(t+\tau)\bigr\|-\| c^{m}f(t)\|\Bigr|\leq \bigl\| f(t +\tau)-c^{m}f(t) \bigr\|,
\end{align*}
the following result simply follows:

\begin{prop}\label{dusanka54321}
Suppose that $f : I \rightarrow E$ is $c$-almost periodic ($c$-uniformly recurrent, semi-$c$-periodic). Then $\|f\| : I\rightarrow [0,\infty)$ is almost periodic (uniformly recurrent, semi-periodic).
\end{prop}

Further on,
we have ($x\in I,$ $\tau>0,$ $l \in {\mathbb N}$):
\begin{align*}
f\bigl(x&+l\tau\bigr)-c^{l}f(x)
\\&=\sum_{j=0}^{l-1}c^{j}\Bigl[ f\bigl(x+(l-j)\tau\bigr)- c f\bigl(x+(l-j-1)\tau\bigr)\Bigr].
\end{align*}
Hence,
$$
\Bigl\|f\bigl(\cdot+l\tau\bigr)-c^{l}f(\cdot)\Bigr\|_{\infty}\leq l\Bigl\|f\bigl(\cdot+\tau\bigr)- c f(\cdot)\Bigr\|_{\infty}.
$$
The above estimate can be used to prove the following:

\begin{prop}\label{uspevanje}
Let $f : I \rightarrow E$ be a $c$-almost periodic function ($c$-uniformly recurrent function, semi-$c$-periodic), and let $l\in {\mathbb N}.$
Then $f(\cdot)$ is $c^{l}$-almost periodic ($c^{l}$-uniformly recurrent, semi-$c^{l}$-periodic).
\end{prop}

Consider now the following condition:
\begin{align}\label{rasta123456}
p\in {\mathbb Z} \setminus \{0\},\ q\in {\mathbb N},\ (p,q)=1 \mbox{ and }\arg(c)=\pi p/q.
\end{align}
The next corollary of Proposition \ref{uspevanje} follows immediately by plugging $l=q:$

\begin{cor}\label{uspeva}
Let $f : I \rightarrow E$ be a continuous function, and let \eqref{rasta123456} hold.
\begin{itemize}
\item[(i)] If $p$ is even and $f(\cdot)$ is $c$-almost periodic ($c$-uniformly recurrent, semi-$c$-periodic), then $f(\cdot)$ is almost periodic (uniformly recurrent, semi-periodic).
\item[(ii)] If $p$ is odd and $f(\cdot)$ is $c$-almost periodic ($c$-uniformly recurrent, semi-$c$-periodic), then $f(\cdot)$ is almost anti-periodic (uniformly anti-recurrent, semi-anti-periodic).
\end{itemize}
\end{cor}

Therefore, if $\arg(c)/\pi \in {\mathbb Q},$ then the class of $c$-almost periodic functions ($c$-uniformly recurrent functions, semi-$c$-periodic functions) is always contained in the class of almost periodic functions (uniformly recurrent functions, semi-periodic functions);
in particular, we have that 
any almost anti-periodic function (uniformly anti-recurrent function, semi-anti-periodic function) is almost periodic (uniformly recurrent, semi-periodic).

Now we will prove the following:

\begin{prop}\label{matsah}
Let $f : I \rightarrow E$ be a continuous function, and let $\arg(c)/\pi \notin {\mathbb Q}.$
\begin{itemize}
\item[(i)] If $f(\cdot)$ is $c$-almost periodic, then $f(\cdot)$ is $c'$-almost periodic for all $c'\in S_{1}.$ 
\item[(ii)] If $f(\cdot)$ is bounded and $c$-uniformly recurrent, then $f(\cdot)$ is $c'$-uniformly recurrent for all $c'\in S_{1}.$ 
\end{itemize}
\end{prop}

\begin{proof}
We will prove only (i). Clearly, it suffices to consider the case in which the function $f(\cdot)$ is not identical to zero.  Let $c'\in S_{1}$ and $\epsilon>0$ be fixed; then the prescribed assumption implies that the set $\{c^{l}: l\in {\mathbb N}\}$ is dense in $S_{1}$ and therefore there exists an increasing sequence  
$(l_{k})$ of positive integers such that 
$\lim_{k\rightarrow +\infty}c^{l_{k}}=c'.$  By Proposition \ref{ogranicenost}, the function $f(\cdot)$ is bounded;
let $k\in {\mathbb N}$ be such that $|c^{l_{k}}-c'|<\epsilon /(2\|f\|_{\infty}),$ and let $\tau>0$ be any $(\epsilon/2,c^{l_{k}})$-period for $f(\cdot).$
Then we have
\begin{align*}
\bigl\| f(x+\tau)-c'f(x)\bigr\| \leq \bigl\| f(x+\tau)-c^{l_{k}}f(x)\bigr\|+\bigl| c^{l_{k}}-c'\bigr| \cdot \|f\|_{\infty}<\epsilon/2+\epsilon/2=\epsilon,
\end{align*}  
for any $x\in I.$
This simply completes the proof.
\end{proof}

\begin{prop}\label{matsahari}
Let $f : I \rightarrow E$ be a continuous function. Then we have the following:
\begin{itemize}
\item[(i)] If $f(\cdot)$ is semi-$c$-periodic and $\arg(c)/\pi \in {\mathbb Q}$, then $f(\cdot)$ is $c'$-almost periodic for all $c'\in \{c^{l} : l\in {\mathbb N}\}.$ 
\item[(ii)] If $f(\cdot)$ semi-$c$-periodic and $\arg(c)/\pi \notin {\mathbb Q}$, then $f(\cdot)$ is $c'$-almost periodic for all $c'\in S_{1}.$ 
\end{itemize}
\end{prop}

\begin{proof}
Let $\epsilon>0$ be fixed. 
To prove (i), it suffices to show that $f(\cdot)$ is $c$-almost periodic (see Proposition \ref{uspevanje}).
Since $\arg(c)/\pi \in {\mathbb Q}$ and \eqref{rasta123456} holds, then we have $c^{1+2lq}=c$ for all $l\in {\mathbb N}.$
Then there exists $p>0$ such that, for every $m\in {\mathbb N}$ and $x\in I,$ we have
$\|f(x+mp)-c^{m}f(x)\|\leq \epsilon. $ With $m=1+2lq, $ we have $\|f(x+(1+2lq)p)-c^{1+2lq}f(x)\|
=\|f(x+(1+2lq)p)-cf(x)\|\leq \epsilon$ so that the conclusion follows from the fact that the set $\{(1+2lq)p : l\in {\mathbb N}\}$
is relatively dense in $[0,\infty).$
Assume now that
$\arg(c)/\pi \notin {\mathbb Q}.$ To prove (ii), it suffices to consider case $f\neq 0.$ Observe first that Lemma \ref{rad98765}
yields that 
there exists a strictly increasing sequence $(l_{k})$ of positive integers such that $\sup_{k\in {\mathbb N}}(l_{k+1}-l_{k})<\infty$ and $|c^{l_{k}}-c'|<\epsilon/\|f\|_{\infty}$
for all $k\in {\mathbb N}.$
With this sequence and the number $p>0$ chosen as above, we have:
\begin{align*}
\bigl\|f(x+pl_{k})-c'f(x)\bigr\|&\leq \bigl\|f(x+pl_{k})-c^{l_{k}}f(x)\bigr\| +\bigl| c^{l_{k}}-c'\bigr| \|f\|_{\infty}
\\& \leq \epsilon +\epsilon \|f\|_{\infty}/\|f\|_{\infty}=2\epsilon,\quad x\in I,\ k\in {\mathbb N}.
\end{align*}
Since the set $\{pl_{k} : k\in {\mathbb N}\}$ is relatively dense in $[0,\infty),$
the proof is completed.
\end{proof}

In connection with Proposition \ref{matsahari}(ii), it is natural to ask whether the assumptions that the function $f(\cdot)$ is semi-$c$-periodic and $\arg(c)/\pi \notin {\mathbb Q}$ imply that $f(\cdot)$ is semi-$c'$-almost periodic for all $c'\in S_{1}?$ 

We continue by providing the
following extension of \cite[Theorem 2.2]{krag} (see also \cite[pp. 3-4]{besik}):

\begin{thm}\label{krewfar}
Let $f  : I\rightarrow E$ be $c$-almost periodic ($c$-uniformly recurrent, semi-$c$-periodic), and let $\alpha \in {\mathbb C}.$ Then we have:
\begin{itemize}
\item[(i)] $\alpha f(\cdot)$ is $c$-almost periodic ($c$-uniformly recurrent, semi-$c$-periodic).
\item[(ii)] If $E={\mathbb C}$ and $\inf_{x\in {\mathbb R}}|f(x)|=m>0,$ then $1/f(\cdot)$ is $1/c$-almost periodic ($1/c$-uniformly recurrent, semi-$1/c$-periodic).
\item[(iii)] If $(g_{n}: I \rightarrow E)_{n\in {\mathbb N}}$ is a sequence of $c$-almost periodic functions ($c$-uniformly recurrent functions, semi-$c$-periodic functions) and $(g_{n})_{n\in {\mathbb N}}$ converges uniformly to a function $g: I \rightarrow E$, then
$g(\cdot)$ is $c$-almost periodic ($c$-uniformly recurrent, semi-$c$-periodic).
\item[(iv)] If $a\in I$ and $b\in I \setminus \{0\},$ then the functions $f(\cdot+a)$ and $f(b\, \cdot)$ are likewise $c$-almost periodic ($c$-uniformly recurrent, semi-$c$-periodic).
\end{itemize}
\end{thm}

A function $f : I \rightarrow E$ is said to be $c$-periodic if and only if there exists $p>0$ such that $f(x+p)=cf(x)$ for all $x\in I.$
Keeping in mind Theorem \ref{krewfar}(iii) and the proofs of \cite[Lemma 1, Theorem 1]{andres1}, we can clarify the following extension of \cite[Proposition 3]{chelj}:

\begin{thm}\label{petruciani}
Let $f\in C_{b}(I :E).$ Then $f(\cdot)$ is semi-$c$-periodic  if and only if  there exists a sequence $(f_{n})$ of $c$-periodic functions in $C_{b}(I :E)$ such that $\lim_{n\rightarrow \infty}f_{n}(x)=f(x)$ uniformly in $I.$   
\end{thm}

We continue by providing two illustrative examples:

\begin{example}\label{kader} (see also \cite[Example 2.2]{krag})
The function $f : {\mathbb R} \rightarrow {\mathbb R}$ given by $f(t):=\cos t,$ $t\in {\mathbb R}$ is $c$-almost periodic ($c$-uniformly recurrent) if and only if $c=\pm 1,$ while $f(\cdot)$ is semi-$c$-periodic if and only if $c=1;$ the function $f_{\varphi} : {\mathbb R} \rightarrow {\mathbb R}$ given by $f_{\varphi}(t):=e^{it\varphi},$ $t\in {\mathbb R}$ ($\varphi \in (-\pi,\pi] \setminus \{0\}$) is $c$-almost periodic (semi-$c$-periodic) for any $c\in S_{1},$
while the function $f_{0}(\cdot)$ is $c$-almost periodic ($c$-uniformly recurrent, semi-$c$-periodic) if and only if $c=1.$
Consider now the function $g : {\mathbb R} \rightarrow {\mathbb R}$ given by $g(t):=2^{-1}\cos 4t +2\cos 2t,$ $t\in {\mathbb R}.$ Then we know that the function $g(\cdot)$ is (almost) periodic and not almost anti-periodic.
Now we will prove that $g(\cdot)$ is $c$-almost periodic ($c$-uniformly recurrent, semi-$c$-periodic) if and only if $c=1.$ Suppose that $(\alpha_{n})$ is a strictly increasing sequence tending to plus infinity such that ($c=e^{i\alpha},$ $\alpha \in (-\pi,\pi]$):
$$
\lim_{n\rightarrow +\infty}\sup_{t\in {\mathbb R}}\Bigl|2^{-1}\cos \bigl(4t +\alpha_{n}\bigr) 2\cos \bigl(2t +\alpha_{n}\bigr)-e^{i\alpha}\bigl[ 2^{-1}\cos 4t +2\cos 2t\bigr] \Bigr|=0.
$$ 
With $t=\pi,$ the above implies
\begin{align}\label{reper}
\lim_{n\rightarrow +\infty}[\cos 4 \alpha_{n}+4\cos 2\alpha_{n}-5\cos \alpha]=0\ \ \mbox{ and }\ \ \lim_{n\rightarrow +\infty}5\sin \alpha=0,
\end{align}
which immediately yields $\alpha=0$ or $\alpha=\pi.$ In the second case, the contradiction is obvious since the first limit equation in \eqref{reper} cannot be fulfilled, while the case $\alpha=0$ is possible and equivalent with the usual almost periodicity of $g(\cdot).$
\end{example}

\begin{example}\label{strina1} (see also \cite[Example 1]{andres1} and \cite[Example 4, Example 5]{chelj})
Let $p$ and $q$ be odd natural numbers such that $p-1 \equiv 0 \  (mod \ q),$ and let $c=e^{i\pi p/q}.$ 
The function
$$
f(x):=\sum_{n=1}^{\infty}\frac{e^{ix/(2nq+1)}}{n^{2}},\quad x\in {\mathbb R} 
$$
is semi-$c$-periodic because it is a uniform limit of $[\pi \cdot (1+2q)\cdot \cdot \cdot (1+2Nq)]$-periodic functions
$$
f_{N}(x):=\sum_{n=1}^{N}\frac{e^{ix/(2nq+1)}}{n^{2}},\quad x\in {\mathbb R} \ \ (N\in {\mathbb N}).
$$
\end{example}

Now we will state and prove the following 

\begin{prop}\label{gruskamiroslavglisiclord}
Suppose that $f : I\rightarrow {\mathbb R}$ is $c$-uniformly recurrent (semi-$c$-periodic) and $f\neq  0.$ Then $c=\pm 1$ and moreover, if $f(t)\geq 0$ for all $t\in I,$ then $c=1.$ 
\end{prop}

\begin{proof}
We will consider the class of $c$-uniformly recurrent functions, only, when
we may assume without loss of generality that $I=[0,\infty).$ Then $f\notin C_{0}([0,\infty) : {\mathbb R});$ namely, if we suppose the contrary, then there exists a strictly increasing sequence $(\alpha_{n})$ of positive real numbers such that $\lim_{n\rightarrow +\infty}\alpha_{n}=+\infty$ and \eqref{rastvor} holds. In particular, for every fixed number $t_{0}\geq 0$ we have
$
\lim_{n\rightarrow +\infty}| f(t_{0} +\alpha_{n})-cf(t_{0})|=0.
$ This automatically yields $f(t_{0})=0$ and, since $t_{0}\geq 0$ was arbitrary, we get $f=0$ identically, which is a contradiction.
Therefore, there exist a strictly increasing sequence $(t_{l})_{l\in {\mathbb N}}$ of positive real numbers tending to plus infinity and a positive real number $a\geq \limsup_{t\rightarrow +\infty}|f(t)|>0$ such that $|f(t_{l})|\geq a/2$ for all $l\in {\mathbb N}.$ Let $\epsilon>0$ be fixed. Then there exist two real numbers $t_{0}>0$ and $n_{0}\in {\mathbb N}$ such that
$|f(t+\alpha_{n})-f(t)|\leq \epsilon$ for all $t\geq t_{0}$ and $n\geq n_{0}.$ If $\arg(c)=\varphi \in (-\pi,\pi],$ then we particularly get that for each $t\geq t_{0}$ and $n\geq n_{0}$ we have:
$$
\bigl| f(t+\alpha_{n})-\cos \varphi \cdot f(t)\bigr| \leq \epsilon \ \ \mbox{ and }\ \ \bigl| \sin \varphi \cdot f(t)\bigr| \leq \epsilon.
$$
Plugging in the second estimate $t=t_{l}$ for a sufficiently large $l\in {\mathbb N}$ we get that $|\sin \varphi| \leq 2\epsilon/a.$ Since $\epsilon>0$ was arbitrary, we get $\sin \varphi=0$ and $c=\pm 1.$ Suppose, finally, that $f(t)\geq 0$ for all $t\geq 0$ and $c=-1.$ Then we have 
$ f(t+\alpha_{n})+f(t)\leq 2\epsilon
$ for all $t\geq t_{0}$ and $n\geq n_{0}.$ Plugging again $t=t_{l}$ for a sufficiently large $l\in {\mathbb N}$ we get that $a\leq \epsilon$ for all $\epsilon>0$ and therefore $a=0,$ which is a contradiction. 
\end{proof}

By the proof of Proposition \ref{gruskamiroslavglisiclord}, we have:

\begin{prop}\label{balija}
Suppose that $f : I\rightarrow E$ is $c$-uniformly recurrent (semi-$c$-periodic) and $f\neq  0.$ Then $f\notin C_{0}(I : E).$
\end{prop}

We continue by providing some illustrative applications of Proposition \ref{gruskamiroslavglisiclord}:

\begin{example}\label{gruskavlade}
\begin{itemize}
\item[(i)]
The function $f: {\mathbb R}\rightarrow {\mathbb R},$ given by
\begin{align}\label{gader}
f(t):=\sum_{n=1}^{\infty}\frac{1}{n}\sin^{2}\Bigl(\frac{t}{2^{n}} \Bigr)\, dt,\quad t\in {\mathbb R},
\end{align}
is unbounded, uniformly continuous and uniformly recurrent (see Haraux and Souplet \cite[Theorem 1.1]{haraux}). Then $f(\cdot)$ is $c$-uniformly recurrent
if and only if $c=1.$ 
\item[(ii)] Let $\tau_{1}:=1,\ \tau_{2}>2$ and let the sequence $(\tau_{n})_{n\in {\mathbb N}}$
of positive real numbers satisfy $\tau_{n}>2\sum_{i=1}^{n-1}i\tau_{i}$ for all $n\in {\mathbb N}.$ 
Let the sequence $(f_{n}: {\mathbb R} \rightarrow {\mathbb R})_{n\in {\mathbb N}}$ be defined as follows. Set $f_{1}(x):=1-|x|$ for $|x|\leq 1$ and $f_{1}(x):=0,$ otherwise. If the functions $f_{1}(\cdot),\cdot \cdot \cdot, f_{n-1}(\cdot)$ are already defined, set
$$
f_{n}(x):=f_{n-1}(x)+\sum_{m=1}^{n-1}\frac{n-m}{n}\Bigl[ f_{n-1}\bigl( x-m\tau_{n}\bigr)+f_{n-1}\bigl( x+m\tau_{n}\bigr) \Bigr],\quad x\in {\mathbb R}.
$$
Then 
$$
\bigl| f_{n}(x+\tau_{n})-f_{n}(x)\bigr| \leq \frac{1}{n},\quad n\in {\mathbb N},\ x\in {\mathbb R},
$$
and the function 
\begin{align}\label{ema}
f(x):=\lim_{n\rightarrow +\infty}f_{n}(x),\quad
x\in {\mathbb R}
\end{align} 
is well defined, even and satisfies that $0\leq f(x)\leq 1$ for all $x\in {\mathbb R}$ (see Bohr, the first part of \cite{h.bor}, pp. 113--115). Then we know that
the function $f : {\mathbb R} \rightarrow {\mathbb R},$ given by \eqref{ema}, is bounded, uniformly continuous and uniformly recurrent (\cite{densities}). Therefore, $f(\cdot)$ is $c$-uniformly recurrent
if and only if $c=1.$
\item[(iii)] The following function has been used by de Vries in \cite[point 6., p. 208]{vries}. 
Let $(p_{i})_{i\in {\mathbb N}}$ be a strictly increasing sequence of natural numbers such that $p_{i}|p_{i+1},$ $i\in {\mathbb N}$ and $\lim_{i\rightarrow \infty}p_{i}/p_{i+1}=0.$ Define the function $f_{i} : [-p_{i},p_{i}]\rightarrow [0,1]$ by $f_{i}(t):=|t|/p_{i},$ $t\in [-p_{i},p_{i}]$ and extend the function $f_{i}(\cdot)$ periodically to the whole real axis; the obtained
function, denoted by the same symbol $f_{i}(\cdot),$ is of period $2p_{i}$ ($i\in {\mathbb N}$). Set 
\begin{align}\label{vries-prcko}
f(t):=\sup\bigl\{f_{i}(t): i\in {\mathbb N}\bigr\},\quad t\in {\mathbb R}.
\end{align}
Then the function $f : {\mathbb R} \rightarrow {\mathbb R},$ given by \eqref{vries-prcko}, is bounded, uniformly continuous and uniformly recurrent (\cite{densities}). Therefore, $f(\cdot)$ is $c$-uniformly recurrent
if and only if $c=1.$ 
\end{itemize}
Any of the above three functions is not asymptotically (Stepanov) almost automorphic (see \cite{densities} for more details). 
\end{example}

The function $f(\cdot)$ constructed in the following example is also not asymptotically (Stepanov) almost automorphic since it is not Stepanov bounded:

\begin{example}\label{dugorocne}
The function $g: {\mathbb R}\rightarrow {\mathbb R},$ given by
\begin{align*}
g(t):=\sum_{n=1}^{\infty}\frac{1}{n}\sin^{2}\Bigl(\frac{t}{3^{n}} \Bigr)\, dt,\quad t\in {\mathbb R},
\end{align*}
is unbounded, 
Lipschitz continuous and uniformly recurrent; furthermore, we have the existence of a positive integer $k_{0}\in {\mathbb N}$ such that 
\begin{align}\label{vladadivljan}
\frac{1}{3^{k}\pi}\int^{3^{k}\pi}_{0}g(s)\, ds \geq \frac{1}{2}\bigl(\ln k-1),\quad k\geq k_{0}
\end{align}
and
\begin{align}\label{majafilip}
\sup_{t\in {\mathbb R}}\bigl| g(t+3^{n}\pi)-g(t) \bigr|\leq \frac{\pi}{n+1}\sum_{j=1}^{\infty}3^{-j},\quad n\in {\mathbb N}.
\end{align}
This can be proved in exactly the same way as in the proof of
\cite[Theorem 1.1]{haraux}. Define now
$f(t):=\sin t \cdot g(t),$ $t\in {\mathbb R}.$ Then \eqref{majafilip} easily implies 
$$
\sup_{t\in {\mathbb R}}\bigl| f(t+3^{n}\pi)+f(t) \bigr|\leq \frac{\pi}{n+1}\sum_{j=1}^{\infty}3^{-j},\quad n\in {\mathbb N}.
$$
Therefore, $f(\cdot)$ is uniformly anti-recurrent and Proposition \ref{gruskamiroslavglisiclord} yields that the function $f(\cdot)$ is $c$-uniformly recurrent if and only if $c=\pm 1.$ To prove that $f(\cdot)$ is Stepanov unbounded, obeserve that \eqref{vladadivljan}
implies the existence of a sequence $(t_{k})_{k\in {\mathbb N}}$ of positive real numbers such that $g(t_{k})\geq (1/2)(\ln k-1)$ for all $k\geq k_{0}. $ If we denote by $L\geq 1$ the Lipschitzian constant of mapping $g(\cdot),$ then the above implies
\begin{align}\label{napoljekerber}
g(x)\geq  (1/2)(\ln k-1)-8L\pi,\quad x\in \bigl[ t_{k},t_{k}+8\pi\bigr],\ k\geq k_{0}.
\end{align}
The existence of a constant $M>0$ such that $\int^{t+1}_{t}|\sin s| \cdot g(s)\, ds<M$ for all $t\in {\mathbb R}$  
would imply by
\eqref{napoljekerber} the existence of a sequence $(a_{k})$ of positive integers such that $[2a_{k}\pi+(\pi/2),2a_{k}\pi+(\pi/2)+1]\subseteq [ t_{k},t_{k}+8\pi\bigr]$ and therefore (take $t=2a_{k}\pi+(\pi/2)$) 
$$
\sin((\pi/2)+1) \cdot \Bigl[ (1/2)(\ln k-1)-8L\pi\Bigr]\leq M,\quad k\geq k_{0},
$$
which is a contradiction.
\end{example}

In connection with Proposition \ref{gruskamiroslavglisiclord} and Proposition \ref{balija}, we would like to present an illustrative example with the complex-valued functions:

\begin{example}\label{brukasham}
Let $h: I \rightarrow {\mathbb R},$ $q: I \rightarrow {\mathbb R}$ and $f(t):=h(t)+iq(t),$ $t\in I.$ Suppose that 
$f : I \rightarrow {\mathbb C}$ is $c$-uniformly recurrent, where
$c=e^{i\varphi}$ and $\sin \varphi \neq 0.$ Then $h\in C_{0}(I: {\mathbb R})$ or $q\in C_{0}(I: {\mathbb R})$ implies $f\equiv 0.$ To show this,
 observe that the $c$-uniform recurrence of $f(\cdot)$ implies the existence of a strictly increasing sequence $(\alpha_{n})$ of positive real numbers tending to plus infinity such that
\begin{align*}
\lim_{n\rightarrow +\infty}&\sup_{t\in I}\bigl| h(t+\alpha_{n})-\cos \varphi \cdot h(t) +\sin \varphi \cdot q(t)\bigr|=0,\mbox{ and }
\\& \lim_{n\rightarrow +\infty}\sup_{t\in I}\bigl| q(t+\alpha_{n})-\cos \varphi \cdot q(t) -\sin \varphi\cdot h(t)\bigr|=0.
\end{align*} 
Since we have assumed that $\sin \varphi \neq 0$, the assumption $h\in C_{0}(I: {\mathbb R})$ ($q\in C_{0}(I: {\mathbb R})$) implies by the first (second) of the above equalities that $q\in C_{0}(I: {\mathbb R})$ ($h\in C_{0}(I: {\mathbb R})$). Hence, $f\in C_{0}(I: {\mathbb C})$ and the claimed statement follows by Proposition \ref{balija}.
\end{example}

The space consisting of all almost periodic functions ($c=1$) is the only function space 
from those introduced in Definition \ref{drmniga}, Definition \ref{onakao} and Definition \ref{semi-ceanti} which
has a linear vector structure:

\begin{example}\label{questionbatty}
\begin{itemize}
\item[(i)]
Suppose that $c=1.$ Then the set of 
all $c$-almost periodic functions is a vector space together with the usual operations,
while the set of $c$-uniformly recurrent functions and the set of semi-$c$-periodic functions are not vector spaces together with the usual operations (\cite{densities}).
\item[(ii)]
Suppose that $c=-1.$ Then the set of 
all $c$-almost periodic functions ($c$-uniformly recurrent functions, semi-$c$-periodic functions) is not a vector space together with the usual operations (\cite{krag}).
\item[(iii)] Suppose that $c\neq \pm 1.$ Then the set of 
all $c$-almost periodic functions ($c$-uniformly recurrent functions, semi-$c$-periodic functions) is not a vector space together with the usual operations. Speaking-matter-of-factly, the functions $f_{\varphi ,\pm} : {\mathbb R} \rightarrow {\mathbb R}$ given by $f_{\varphi,\pm}(t):=e^{\pm it\varphi},$ $t\in {\mathbb R}$ ($\varphi \in (-\pi,\pi] \setminus \{0\}$) are $c$-almost periodic (semi-$c$-periodic); see Example \ref{kader}.
Its sum $f_{\varphi ,+}(\cdot)+f_{\varphi ,-}(\cdot)=2\cos \varphi \cdot$ is not $c$-uniformly recurrent due to Proposition \ref{gruskamiroslavglisiclord}.  
\end{itemize}
\end{example}

Similarly, we have:

\begin{example}\label{questionbatty12345}
Let $f : I\rightarrow {\mathbb C}$ and $g: I\rightarrow E.$ 
\begin{itemize}
\item[(i)]
Suppose that $c=1.$ If $f\in AP(I:{\mathbb C})$ and $g\in AP(I: E),$ then $f\cdot g \in AP(I:E);$
furthermore, there exist $f\in UR(I:{\mathbb C})$ and $g\in UR(I: E)$ such that $f\cdot g \notin UR(I:E)$  (\cite{densities}). It can be simply proved that the pointwise product of anti-periodic functions $f(t):=\cos t,$ $t\in {\mathbb R}$ and $g(t):=\cos \sqrt{2}t,$ $t\in {\mathbb R}$ is not a semi-periodic function (see e.g., \cite[Lemma 2]{andres1}).
\item[(ii)]
Suppose that $c=-1.$ Then there exist an anti-periodic function $f(\cdot)$ and an anti-periodic function $g(\cdot)$ such that $f\cdot g(\cdot)$ is not anti-uniformly recurrent.
We can simply take $E={\mathbb C}$ and $f(t):=g(t):=\cos t,$ $t\in I.$
\item[(iii)] Suppose that $c\neq \pm 1.$ Then 
there exist a semi-$c$-periodic function $f(\cdot)$ and a semi-$c$-periodic function $g(\cdot)$ such that $f\cdot g(\cdot)$ is not $c$-uniformly recurrent. 
Consider again the functions $f_{\varphi ,\pm} : {\mathbb R} \rightarrow {\mathbb R}$ given by $f_{\varphi,\pm}(t):=e^{\pm it\varphi},$ $t\in {\mathbb R}$ ($\varphi \in (-\pi,\pi] \setminus \{0\}$). They are semi-$c$-periodic  but their pointwise product 
$f_{\varphi ,+}(\cdot)\cdot f_{\varphi ,-}(\cdot)=1$ is not $c$-uniformly recurrent due to Proposition \ref{gruskamiroslavglisiclord}.  
\end{itemize}
\end{example}

Denote by $ANP_{0}(I : E)$ and $ANP(I : E)$ the linear span of almost anti-periodic functions $I \mapsto E$ and its closure  in $AP(I : E),$
respectively. In \cite[Theorem 2.3]{krag}, we have shown that
$ANP(I : E)=AP_{{\mathbb R} \setminus \{0\}}(I : E).$ Now we will extend this result in the following way:

\begin{thm}\label{deckoracunao}
Denote by $AP_{c,0}(I : E)$ and $\mathbf{AP}_{c,0}(I : E)$  the linear span of $c$-almost periodic functions $I \mapsto E$ and its closure  in $AP(I : E),$
respectively. Then the following holds:
\begin{itemize}
\item[(i)] Let $\arg(c)\in \pi \cdot {\mathbb Q}.$ 
Then we have $\mathbf{AP}_{c,0}(I : E)=AP_{{\mathbb R} \setminus \{0\}}(I : E).$ 
\item[(ii)] Let  $\arg(c)\notin \pi \cdot {\mathbb Q}.$ 
Then we have $\mathbf{AP}_{c,0}(I : E) \supseteq AP_{{\mathbb R} \setminus \{0\}}(I : E).$ 
\end{itemize}
\end{thm}

\begin{proof}
Assume first that
$f\in AP_{{\mathbb R} \setminus \{0\}}(I : E).$
By \eqref{zaq54321}, we have 
$$
f\in \overline{span\{e^{i\mu \cdot} x : \mu \in \sigma(f),\ x\in R(f)\}},
$$
where the closure is taken in the space $C_{b}(I : E).$
Since $\sigma(f) \subseteq {\mathbb R} \setminus \{0\}$ and the function $t\mapsto e^{i\mu t},$ $t\in I$ ($\mu \in {\mathbb R} \setminus \{0\}$) is $c$-almost periodic for all $c\in S_{1},$ we have that $span\{e^{i\mu \cdot} x : \mu \in \sigma(f),\ x\in R(f)\}\subseteq AP_{c,0}(I : E).$ Hence, $f\in \mathbf{AP}_{c,0}(I : E).$ To complete the proof, it remains to
consider case $\arg(c)\in \pi \cdot {\mathbb Q}$ and show that any  
function $f\in \mathbf{AP}_{c,0}(I : E)$ belongs to the space $AP_{{\mathbb R} \setminus \{0\}}(I : E).$
Furthermore, it suffices to consider case in which
\eqref{rasta123456} holds with the number $p$ even because otherwise we can apply Corollary \ref{uspeva}(ii) and Proposition \ref{matsah}(i) to see that 
$AP_{c,0}(I : E)\subseteq ANP_{0}(I : E)$ and therefore $\mathbf{AP}_{c,0}(I : E) \subseteq ANP(I : E),$ so that the statement directly follows from \cite[Theorem 2.3]{krag}.
We will prove that
\begin{align}\label{spektar}
\lim_{t\rightarrow \infty}\frac{1}{t}\int^{t}_{0}f(s)\, ds =0;
\end{align}
clearly, by almost periodicity of $f(\cdot),$ the limit in \eqref{spektar} exists. Let $\epsilon>0$ be fixed, and let $l>0$ satisfy that every interval of $[0,\infty)$ of length $l$ contains a point $\tau $ such that $\|f(t+\tau)-cf(t)\|\leq \epsilon,$ $t\geq 0.$
We have $c^{q}=1$ and therefore $1+c+\cdot \cdot \cdot +c^{q-1}=0;$ using this equality and decomposition ($s\geq 0,\ n\in {\mathbb N}$)
\begin{align*}
\bigl\|f(s&+(n
-1)\tau)+f(s+(n-2)\tau)+\cdot \cdot \cdot +f(s)\bigr\|
\\&\leq \epsilon +\bigl\|(1+c)f(s+(n-2)\tau)+f(s+(n-3)\tau)+\cdot \cdot \cdot +f(s)\bigr\|
\\& \leq \epsilon +\bigl\|(1+c)f(s+(n-2)\tau)-(1+c)cf(s+(n-3)\tau)\bigr\|
\\&+\bigl\|[1+(1+c)c]f(s+(n-3)\tau)+f(s+(n-4)\tau)\cdot \cdot \cdot+ f(s)\bigr\|
\\& \leq \epsilon +|1+c|\epsilon +\Bigl\|[1+c+c^{2}]f(s+(n-3)\tau)+f(s+(n-4)\tau)+\cdot \cdot \cdot +f(s)\Bigr\|
\\& \leq  \epsilon +|1+c|\epsilon +\bigl|1+c+c^{2}\bigr|\epsilon +...
\\& \leq  \epsilon +|1+c|\epsilon +\bigl|1+c+c^{2}\bigr|\epsilon +...+\bigl| 1+c+c^{2}+\cdot \cdot \cdot+c^{q-2}\bigr|\epsilon
\\& +\bigl\| f(s)+f(s+\tau)+\cdot \cdot \cdot +f(s+(n-1-q)\tau) \bigr\|,
\end{align*}
we immediately get that there exists a finite constant $A\geq 1$ such that, for every $s\geq 0$ and $ n\in {\mathbb N},$
\begin{align*}
\bigl\|f(s&+(n
-1)\tau)+f(s+(n-2)\tau)+\cdot \cdot \cdot +f(s)\bigr\|\leq A\epsilon\lceil n/q\rceil +A\|f\|_{\infty}.
\end{align*}
Integrating this estimate over the segment $[0,n\tau],$ we get that, for every $s\geq 0$ and $ n\in {\mathbb N},$
\begin{align*}
\Biggl\|\int^{n\tau}_{0}f(s)\, ds\Biggr\| &=\Biggl\|\int^{\tau}_{0}\bigl[f(s+(n
-1)\tau)+f(s+(n-2)\tau)+\cdot \cdot \cdot +f(s) \bigr]\, ds\Biggr\|
\\& \leq A\tau\epsilon\lceil n/q\rceil +A\tau\|f\|_{\infty}.
\end{align*}
Dividing the both sides of the above inequality with $n\tau$, we get that 
$$
\lim_{n\rightarrow +\infty}\Biggl\|\frac{1}{n\tau}\int^{n\tau}_{0}f(s)\, ds\Biggr\| \leq A\epsilon /q.
$$
Since $\epsilon>0$ was arbitrary, this immediately yields \eqref{spektar}.
\end{proof}

Now we will state and prove the following result: 

\begin{prop}\label{batty}
Suppose that $f : [0,\infty) \rightarrow E$ is $c$-almost periodic (semi-$c$-periodic). Then ${\mathbf E}f: {\mathbb R} \rightarrow E$ is a unique $c$-almost periodic extension (semi-$c$-periodic extension)
of $f(\cdot)$ to the whole real axis.
\end{prop}

\begin{proof}
The proof for the class of $c$-almost periodic functions is very similar to the proof of \cite[Proposition 2.2]{krag} and therefore omitted. For the class of semi-$c$-periodic functions, the proof can be deduced as follows (see also \cite[Proposition 4]{chelj}). 
Due to Proposition \ref{matsahari}, we have that the function $f: [0,\infty) \rightarrow E$ is almost periodic, so that the function ${\mathbf E}f: {\mathbb R} \rightarrow E$ is a unique almost periodic extension of $f(\cdot)$ to the whole real axis.
Therefore, it remains to be proved that ${\mathbf E}f(\cdot)$ is semi-$c$-periodic. Let $\epsilon>0$ be fixed. Then there exists $p>0$ such that for all $m\in {\mathbb N}$ and $x\geq 0$ we have $\|f(x+mp)-c^{m}f(x)\|\leq \epsilon.$
For every fixed number $m\in {\mathbb N},$ the function ${\mathbf E}f(\cdot+mp)-c^{m}{\mathbf E}f(\cdot)$ is almost periodic so that the supremum formula implies
\begin{align*}
\sup_{x\in {\mathbb R}}\bigl\| {\mathbf E}f(x+mp)-c^{m}{\mathbf E}f(x) \bigr\| &=\sup_{x\geq 0}\bigl\| {\mathbf E}f(x+mp)-c^{m}{\mathbf E}f(x) \bigr\|  
\\ &= \sup_{x\geq 0}\bigl\| f(x+mp)-c^{m}f(x) \bigr\| 
\leq \epsilon.
\end{align*}
This completes the proof.
\end{proof}

We continue by introducing the following notion:

\begin{defn}\label{asymp-cea}\index{function!asymptotically $c$-uniformly recurrent}\index{function!asymptotically $c$-almost periodic}\index{function!asymptotically semi-$c$-periodic}
A continuous function $f: I \rightarrow E$ is called asymptotically $c$-uniformly recurrent (asymptotically $c$-almost periodic, asymptotically semi-$c$-periodic) if and only if there are a $c$-uniformly recurrent ($c$-almost periodic, semi-$c$-periodic) function $g: I \rightarrow E$ and a function $h\in C_{0}(I : E)$ such that 
$
f(x)=g(x)+h(x),$ $x\in I.$
\end{defn}

%Decomposition $
%f(x)=g(x)+h(x),$ $x\in I$ of function $f(\cdot)$ into its principal part $g(\cdot)$ and its corrective (ergodic) part $q(\cdot)$ is unique due %to Proposition \ref{balija}.

For the Stepanov  classes, we will use the following notion:

\begin{defn}\label{stepa-anti-c}\index{function!}\index{function!Stepanov $(p,c)$-uniformly recurrent}\index{function!Stepanov $(p,c)$-almost periodic}\label{stepa-anti-c}\index{function!asymptotically Stepanov $(p,c)$-uniformly recurrent}\index{function!asymptotically Stepanov $(p,c)$-almost periodic}
Let $1\leq p<\infty,$ and let $f\in L_{loc}^{p}(I : E).$
\begin{itemize}
\item[(i)] It is said that $f(\cdot)$ is Stepanov $(p,c)$-uniformly recurrent (Stepanov $(p,c)$-almost periodic, Stepanov semi-$(p,c)$-periodic) if and only if the function $\hat{f}:  I \rightarrow L^{p}([0,1]:E),$ defined by 
\eqref{zad},
is $c$-uniformly recurrent ($c$-almost periodic, semi-$c$-periodic).
\item[(ii)] It is said that $f(\cdot)$ is asymptotically Stepanov $(p,c)$-uniformly recurrent (asymptotically Stepanov $(p,c)$-almost periodic, asymptotically Stepanov semi-$(p,c)$-periodic) if and only if there are a Stepanov $(p,c)$-uniformly 
recurrent (Stepanov $(p,c)$-almost periodic, Stepanov semi-$(p,c)$-periodic) function $h(\cdot)$ and $q\in C_{0}(I : L^{p}([0,1] : E))$ such that
$f(t)=h(t)+q(t)$ for a.e. $t\in I.$
\end{itemize}
\end{defn}  

\subsection{Composition principles for $c$-almost periodic type functions}\label{poredfgha}

In this subsection, we will clarify and prove several composition principles for $c$-almost periodic functions and $c$-uniformly recurrent functions; the composition 
theorems for semi-$c$-periodic functions will be investigated in \cite{novipinto}. 

Suppose that $F : I\times X\rightarrow E$ is a continuous function and there exists a finite constant $L>0$ such that
\begin{align}\label{lipshic}
\| F(t,x)-F(t,y)\|\leq L\|x-y\|_{X},\quad t\in I,\ x,\ y\in X.
\end{align}
Define ${\mathcal F}(t):=F(t,f(t)),$ $t\in I.$ 
We need the following estimates ($\tau \geq 0,$ $c\in {\mathbb C}\setminus\{0\},$ $t\in I$):
\begin{align}
\notag &\Bigl\| F(t+\tau,f(t+\tau))-cF(t,f(t))\Bigr\|
\\ \notag &\leq  \Bigl\| F(t+\tau,f(t+\tau))-  F\Bigl(t+\tau,cf(t)\Bigr)\Bigl\| + \Bigl\| F\Bigl(t+\tau,cf(t)\Bigr)-cF(t,f(t))\Bigl\| 
\\\label{krujicalno}& \leq L\Bigl\|f(t+\tau)- cf(t)\Bigr\|+\Bigl\| F\Bigl(t+\tau,cf(t)\Bigr)-cF(t,f(t))\Bigl\| .
\end{align}

Using \eqref{krujicalno}, we can simply deduce the following result:

\begin{thm}\label{prcko-radojica-1234567890}
Suppose that $F : I\times X\rightarrow E$ is a continuous function and there exists a finite constant $L>0$ such that
\eqref{lipshic} holds. 
\begin{itemize}
\item[(i)] 
Suppose that $f : I \rightarrow X$ is $c$-uniformly recurrent. 
If there exists a strictly increasing sequence $(\alpha_{n})$ of positive reals tending to plus infinity such that 
\begin{align}\label{jaroslav12345}
\lim_{n\rightarrow+\infty}\sup_{t\in I}\bigl\|f(t+\alpha_{n})-c f(t)\bigr\|=0
\end{align}
and
\begin{align}\label{dragoslav12345}
\lim_{n\rightarrow +\infty}\sup_{t\in I}\Bigl\| F\Bigl(t+\alpha_{n},cf(t)\Bigr)-cF(t,f(t))\Bigl\|=0, 
\end{align}
then the mapping ${\mathcal F}(t):=F(t,f(t)),$ $t\in I$ is $c$-uniformly recurrent.
\item[(ii)] Suppose that $f : I \rightarrow X$ is $c$-almost periodic.
If for each $\epsilon>0$ the set of all positive real numbers $\tau>0$ such that
\begin{align}\label{jaroslav123456}
\sup_{t\in I}\bigl\|f(t+\tau)-cf(t)\bigr\|<\epsilon
\end{align}
and
\begin{align}\label{jaroslav12345666678}
\sup_{t\in I}\Bigl\| F\Bigl(t+\tau,cf(t)\Bigr)-cF(t,f(t))\Bigl\|<\epsilon, 
\end{align}
is relatively dense in $[0,\infty),$
then the mapping ${\mathcal F}(t):=F(t,f(t)),$ $t\in I$ is $c$-almost periodic.
\end{itemize}
\end{thm}

For the class of asymptotically $c$-almost periodic functions, the following result simply follows from the previous theorem and the argumentation used in the proof of \cite[Theorem 3.49]{diagana}:

\begin{thm}\label{prcko-radojica-12345678901}
Suppose that $F : I\times X\rightarrow E$ and  
$Q : I\times X\rightarrow E$
are continuous functions and there exists a finite constant $L>0$ such that
\eqref{lipshic} holds as well as that \eqref{lipshic} holds with the function $F(\cdot,\cdot)$ replaced therein with the function
$Q(\cdot,\cdot).$
\begin{itemize}
\item[(i)] 
Suppose that $g: I \rightarrow E$ is a $c$-uniformly recurrent function, $h\in C_{0}(I : E)$ and
$
f(x)=g(x)+h(x),$ $x\in I.$
If there exists a strictly increasing sequence $(\alpha_{n})$ of positive reals tending to plus infinity such that \eqref{jaroslav12345}
and
\eqref{dragoslav12345}
hold with the function $f(\cdot)$ replaced therein with the function $g(\cdot)$,
$\lim_{|t|\rightarrow +\infty}Q(t, y) = 0$ uniformly for $y\in R(f),$
then the mapping ${\mathcal H}(t):=(F+Q)(t,f(t)),$ $t\in I$ is asymptotically $c$-uniformly recurrent.
\item[(ii)] Suppose that $g: I \rightarrow E$ is a $c$-almost periodic function, $h\in C_{0}(I : E)$ and
$
f(x)=g(x)+h(x),$ $x\in I.$
If for each $\epsilon>0$ the set of all positive real numbers $\tau>0$ such that \eqref{jaroslav123456}
and
\eqref{jaroslav12345666678}
hold with the function $f(\cdot)$ replaced therein with the function $g(\cdot)$,
$\lim_{|t|\rightarrow +\infty}Q(t, y) = 0$ uniformly for $y\in R(f),$
then the mapping ${\mathcal H}(t):=(F+Q)(t,f(t)),$ $t\in I$ is asymptotically $c$-almost periodic.
\end{itemize}
\end{thm}

For the Stepanov classes, the following result slightly generalizes the well known result of Long and Ding \cite[Theorem 2.2]{comp-adv}. The proof can be deduced by using the argumentation used in the proofs of the above-mentioned theorem and  \cite[Theorem 3.18]{abdoje}:

\begin{thm}\label{djurologijamatre}
Let 
$p,\ q\in [1,\infty) ,$ $r\in [1,\infty],$ $1/p=1/q+1/r$ and the following conditions hold:
\begin{itemize}
\item[(i)] Let $F : I \times X \rightarrow E$ 
and let there exist a function $ L_{F}\in L_{S}^{r}(I) $ such that
\begin{align}\label{vbnmp-manijak}
\| F(t,x)-F(t,y)\| \leq L_{F}(t)\|x-y\|_{X},\quad t\in I,\ x,\ y\in X.
\end{align}
\item[(ii)] There exists a strictly increasing sequence $(\alpha_{n})$ of positive real numbers tending to plus infinity such that 
\begin{align}\label{djurociljvjet}
\lim_{n\rightarrow +\infty}\sup_{t\in I}\sup_{u\in R(f)} \int^{t+1}_{t} \Bigl\|F\bigl(s+\alpha_{n},cu\bigr)-cF(s,u)\Bigr\|^{p}\, ds=0
\end{align}
and
\begin{align}\label{djurociljvjet12345}
\lim_{n\rightarrow+\infty}\sup_{t\in I}\int^{t+1}_{t}\Bigl\| f(s+\alpha_{n})-cf(s)\Bigr\|^{q}\, ds=0.
\end{align}
\end{itemize}
Then the function $F(\cdot, f(\cdot))$ is Stepanov $(p,c)$-uniformly recurrent. Furthermore, the assumption that $F(\cdot,0)$ is Stepanov $p$-bounded implies that the function 
$F(\cdot, f(\cdot))$ is Stepanov $p$-bounded, as well.
\end{thm}

Similarly, we can prove the following

\begin{thm}\label{djurologijaforemate}
Suppose $p > 1 $
and the following conditions hold:
\begin{itemize}
\item[(i)] Let $F : I \times X \rightarrow E$ and let there exist a number  $ r\geq \max (p, p/p -1)$ and a function $ L_{F}\in L_{S}^{r}(I) $ such that
\eqref{vbnmp-manijak} holds.
\item[(ii)] There exists a strictly increasing sequence $(\alpha_{n})$ of positive real numbers tending to plus infinity such that 
\eqref{djurociljvjet} holds
and
\eqref{djurociljvjet12345} holds with the number $q$ replaced by the number $p$ therein.
\end{itemize}
Then $q:=pr/(p+r) \in [1, p)$ and the function $F(\cdot, f(\cdot))$ is Stepanov $(q,c)$-uniformly recurrent. Furthermore, the assumption that $F(\cdot,0)$ is Stepanov $q$-bounded implies that the function 
$F(\cdot, f(\cdot))$ is Stepanov $q$-bounded, as well.
\end{thm}

The above results can be simply reformulated for the class of Stepanov $(p,c)$-almost periodic functions.
For the classes of asymptotically Stepanov $(p,c)$-uniformly recurrent (asymptotically Stepanov $(p,c)$-almost periodic) functions, we can 
simply extend the assertions of \cite[Proposition 2.7.3, Proposition 2.7.4]{nova-mono}. Details can be left to the interested readers.

\subsection{Invariance of $c$-almost type periodicity under the actions of convolution products}\label{stepamojalo}

In this subsection, we investigate the invariance of $c$-almost periodicity, $c$-uniform recurrence and semi-$c$-periodicity under the actions of finite and  infinite convolution products. 

We start by stating the following slight generalizations of \cite[Proposition 3.1, Proposition 3.2]{krag}, which can be deduced by using almost the same arguments as in this paper (see also \cite[Proposition 3.1]{densities} and \cite[Proposition 6, Proposition 7]{chelj}):

\begin{prop}\label{cuj-rad}
Suppose $1\leq p <\infty,$ $1/p +1/q=1$
and $(R(t))_{t> 0}\subseteq L(E,X)$ is a strongly continuous operator family satisfying that $M:=\sum_{k=0}^{\infty}\|R(\cdot)\|_{L^{q}[k,k+1]}<\infty .$ If $f : {\mathbb R} \rightarrow E$ is Stepanov $(p,c)$-almost periodic (Stepanov $p$-bounded and Stepanov $(p,c)$-uniformly recurrent/Stepanov $p$-bounded and Stepanov semi-$(p,c)$-periodic), then the function $F(\cdot),$ given by
\begin{align}\label{zad123}
F(t):=\int^{t}_{-\infty}R(t-s)f(s)\, ds,\quad t \in {\mathbb R},
\end{align}
is well-defined and $c$-almost periodic (bounded $c$-uniformly recurrent/bounded and semi-$c$-periodic).
\end{prop}

\begin{prop}\label{stewpa-wqer}
Suppose $1\leq p <\infty,$ $1/p +1/q=1$
and $(R(t))_{t> 0}\subseteq L(E,X)$ is a strongly continuous operator family satisfying that, for every 
$s\geq 0,$ we have 
$$
m_{s}:=\sum_{k=0}^{\infty}\|R(\cdot)\|_{L^{q}[s+k,s+k+1]}<\infty .
$$ 
Suppose, further, that $g: {\mathbb R} \rightarrow E$ is Stepanov $(p,c)$-almost periodic (Stepanov $p$-bounded and Stepanov $(p,c)$-uniformly recurrent/Stepanov $p$-bounded and Stepanov semi-$(p,c)$-periodic), $w : [0,\infty)\rightarrow E$ satisfies
$\hat{w}\in C_{0}([0,\infty) : L^{p}([0,1]:E))$ and $f(t)=g(t)+w(t)$ for all $t\geq 0.$
Let there exist a finite number $M >0$ such that
the following holds:
\begin{itemize}
\item[(i)] $ \lim_{t\rightarrow +\infty}\int^{t+1}_{t}\bigl[\int_{M}^{s}\|R(r)\| \|w(s-r)\| \, dr\bigr]^{p}\, ds=0.$ 
\item[(ii)] $\lim_{t\rightarrow +\infty}\int^{t+1}_{t}m_{s}^{p}\, ds=0.$
\end{itemize}
Then the function $H(\cdot),$ given by
\begin{align*}
H(t):=\int^{t}_{0}R(t-s)f(s)\, ds,\quad t\geq 0,
\end{align*}
is well-defined, bounded and asymptotically Stepanov $(p,c)$-almost periodic (asymptotically Stepanov $(p,c)$-uniformly recurrent/asymptotically Stepanov semi-$(p,c)$-periodic). 
\end{prop}

In the following slight extension of \cite[Proposition 3.2]{densities}, we consider the case in which the forcing term $f(\cdot)$ is not Stepanov $p$-bounded, in general (the proof is essentially the same and therefore omitted; Proposition \ref{stewpa-wqer} can be reformulated in this context, as well):

\begin{prop}\label{redari12345}
Suppose that $1\leq p <\infty,$ $1/p +1/q=1,$ $f : {\mathbb R} \rightarrow E$ is Stepanov $(p,c)$-almost periodic (Stepanov $(p,c)$-uniformly recurrent/Stepanov semi-$(p,c)$-periodic), there exists a continuous function $P: {\mathbb R} \rightarrow [1,\infty)$ such that 
\begin{align*}
\Biggl(\int^{t+1}_{t}\|f(s)\|^{p}\, ds\Biggr)^{1/p} \leq P(t),\quad t\in {\mathbb R}
\end{align*} 
and $(R(t))_{t> 0}\subseteq L(E,X)$ is a strongly continuous operator family satisfying that for each $t\in {\mathbb R}$ we have
$$
\sum_{k=0}^{\infty}\|R(\cdot)\|_{L^{q}[k,k+1]}P(t-k)<\infty .
$$ 
If the function $\hat{f} : {\mathbb R} \rightarrow L^{p}([0,1] : E)$ is uniformly continuous, then the function $F: {\mathbb R} \rightarrow X,$ given by
\eqref{zad123},
is well-defined and $c$-almost periodic ($c$-uniformly recurrent/semi-$c$-periodic).
\end{prop}

\section{Applications to the abstract Volterra integro-differential inclusions}\label{profice97531} 

In this section, we will present some illustrative applications of our abstract results in the analysis of the existence and uniqueness of $c$-almost periodic type solutions to the abstract (semilinear) Volterra integro-differential inclusions. 

First of all, 
we would like to note that the results established in Subsection \ref{stepamojalo} can be applied at any place where the variation of parameters formula takes effect.
Concerning semilinear problems, we can apply
our results in the study of the existence and uniqueness of $c$-almost periodic solutions and $c$-uniformly recurrent solutions of the fractional semilinear Cauchy 
inclusion
\begin{align}\label{mqwert12345}
D_{t,+}^{\gamma}u(t)\in {\mathcal A}u(t)+F(t,u(t)),\ t\in {\mathbb R},
\end{align}
where $D_{t,+}^{\gamma}$ denotes the Riemann-Liouville fractional derivative of order $\gamma \in (0,1),$ 
$F : {\mathbb R} \times X \rightarrow E$ satisfies certain properties, and ${\mathcal A}$ is a closed multivalued linear operator satisfying condition \cite[(P)]{nova-mono}. 
To explain this in more detail, fix a strictly increasing sequence $(\alpha_{n})$ of positive reals tending to plus infinity and define
\begin{align*}
BUR_{(\alpha_{n});c}({\mathbb R} : E)&:=\Bigl\{f\in UR_{c}({\mathbb R} : E) \, ;\, f(\cdot)\mbox{ is bounded and } 
\\&
\lim_{n\rightarrow+\infty}\sup_{t\in {\mathbb R}}\bigl\| f(t+\alpha_{n})-cf(t)\bigr\|_{\infty}=0
\Bigr\}.
\end{align*}
Equipped with the metric $d(\cdot,\cdot):=\|\cdot-\cdot\|_{\infty},$ $BUR_{(\alpha_{n});c}({\mathbb R} : E)$ becomes a complete metric space. Let $(R_{\gamma}(t))_{t>0}$ be the operator family considered in \cite{nova-mono} and \cite{fcaj}. Then we know that
\begin{align}\label{deran98765}
\|R_{\gamma}(t)\|=O\bigl(t^{\gamma-1}+t^{-\gamma-1}\bigr),\quad t>0.
\end{align}
It is said that a continuous function $u: {\mathbb R} \rightarrow E$ is a mild solution of \eqref{mqwert12345} if and only if
\begin{align*}
u(t)=\int^{t}_{-\infty}R_{\gamma}(t-s)F\bigl(s,u(s)\bigr)\, ds,\quad t\in {\mathbb R}.
\end{align*}

Now we are able to state the following result, which is very similar to \cite[Theorem 3.1]{nova-mono}:

\begin{thm}\label{valjevosabac}
Suppose that the function $F : {\mathbb R} \times E \rightarrow E$ satisfies that for each bounded subset $B$ of $E$ there exists a finite real constant $M_{B}>0$ such that
$\sup_{t\in {\mathbb R}}\sup_{y\in B}\| F(t,y)\|\leq M_{B}.$ Suppose, further, that
the function $F : {\mathbb R}\times E \rightarrow E$ is Stepanov $(p,c)$-uniformly recurrent with  $p > 1, $ and there exist a number $ r\geq \max (p, p/p -1)$ and a function $ L_{F}\in L_{S}^{r}(I) $ such that
$q:=pr/(p+r)>1$ and \eqref{vbnmp-manijak} holds with $I={\mathbb R}.$ If
\begin{align}\label{gejacina}
\frac{(\gamma -1)q}{q-1}>-1,
\end{align}
there exists an integer $n\in {\mathbb N}$ such that
$M_{n}<1,$ where 
\begin{align*}
M_{n}:=&\sup_{t\geq 0}\int^{t}_{-\infty}\int^{x_{n}}_{-\infty}\cdot \cdot \cdot \int^{x_{2}}_{-\infty}
\Bigl\| R_{\gamma}(t-x_{n})\Bigr\|
\\& \times \prod^{n}_{i=2}\Bigl\| R_{\gamma}(x_{i}-x_{i-1})\Bigr\| \prod^{n}_{i=1}L_{F}(x_{i})\, dx_{1}\, dx_{2}\cdot \cdot \cdot \, dx_{n},
\end{align*}
and 
\eqref{djurociljvjet} holds with the set $R(f)$ replaced therein with an arbitrary bounded set $B\subseteq E,$ then the abstract semilinear fractional Cauchy inclusion \eqref{mqwert12345} has a unique bounded uniformly recurrent solution which belongs to the space $BUR_{(\alpha_{n}); c}({\mathbb R} : E). $
\end{thm}

\begin{proof}
Define $\Upsilon : BUR_{(\alpha_{n});c}({\mathbb R} : E) \rightarrow BUR_{(\alpha_{n});c}({\mathbb R} : E)$ by
$$
(\Upsilon u)(t):=\int^{t}_{-\infty}R_{\gamma}(t-s)F(s,u(s))\, ds,\quad t\in {\mathbb R}.
$$
 Suppose that $u\in BUR_{(\alpha_{n});c}({\mathbb R} : E).$ Then $R(u)=B$ is a bounded set and the mapping $t\mapsto F(t,u(t)),$ $t\in {\mathbb R}$ is bounded due to the prescribed assumption.
Applying Theorem \ref{djurologijaforemate}, we have that the function $F(\cdot, u(\cdot))$ is Stepanov $(q,c)$-uniformly recurrent. 
Define $q':=q/(q-1).$ By \eqref{deran98765} and \eqref{gejacina}, we have $\|R_{\gamma}(\cdot)\| \in L^{q'}[0,1]$ and
$
\sum_{k=0}^{\infty}\| R_{\gamma}(\cdot)\|_{L^{q'}[k,k+1]}<\infty .
$
Applying Proposition \ref{cuj-rad}, 
we get that the function $t\mapsto \int^{t}_{-\infty}R_{\gamma}(t-s)F(s,u(s))\, ds,$ $t\in {\mathbb R}$
is bounded and $c$-uniformly recurrent, implying that $
\Upsilon u \in BUR_{(\alpha_{n});c}({\mathbb R} : E) $, as claimed.
Furthermore, a simple calculation shows that
\begin{align*}
\Bigl \| \bigl(\Upsilon^{n} u_{1}\bigr)-\bigl(\Upsilon^{n} u_{2} \bigr)\Bigr\|_{\infty}\leq M_{n}\bigl\| u_{1}-u_{2}\bigr\|_{\infty},\quad u_{1},\ u_{2}\in BUR_{(\alpha_{n});c}({\mathbb R} : E),\ n\in {\mathbb N}.
\end{align*}
Since there exists an integer $n\in {\mathbb N}$ such that
$M_{n}<1,$ the well known extension of the Banach contraction principle shows that the mapping $\Upsilon (\cdot)$ has a unique fixed point, finishing the proof of the theorem.
\end{proof}

Similarly we can analyze the existence and uniqueness of asymptotically
Stepanov $(p,c)$-almost periodic
solutions and Stepanov $(p,c)$-uniformly recurrent solutions of the fractional semilinear Cauchy inclusion
\[
\hbox{(DFP)}_{f,\gamma} : \left\{
\begin{array}{l}
{\mathbf D}_{t}^{\gamma}u(t)\in {\mathcal A}u(t)+F(t,u(t)),\ t\geq 0,\\
\quad u(0)=x_{0},
\end{array}
\right.
\]
where ${\mathbf D}_{t}^{\gamma}$ denotes the Caputo fractional derivative of order $\gamma \in (0,1],$ $x_{0}\in E$ and
$F : [0,\infty) \times X \rightarrow E,$ satisfies certain properties, and ${\mathcal A}$ is a closed multivalued linear operator satisfying condition \cite[(P)]{nova-mono}.

The examples and results presented by Zaidman \cite[Examples 4, 5, 7, 8; pp. 32-34]{30} can be used to provide 
certain applications of our results, as well. For example, the unique regular solution of the heat equation $u_{t}(x,t)=u_{xx}(x,t),$ $x\in {\mathbb R},$ $t\geq 0,$ accompanied with the initial condition $u(x,0)=f(x),$ is given by
$$
u(x,t):=\frac{1}{2\sqrt{\pi t}}\int^{+\infty}_{-\infty}e^{-\frac{(x-s)^{2}}{4t}}f(s)\, ds,\quad x\in {\mathbb R},\ t>0;
$$
see \cite[Example 4]{30}.
Let the number $t_{0}>0$ be fixed, and let the function $f(\cdot)$ be bounded $c$-uniformly recurrent ($c$-almost periodic, semi-$c$-periodic). Since $e^{-\cdot^{2}/4t_{0}} \in L^{1}({\mathbb R}),$ we can use the fact that the space of bounded $c$-uniformly recurrent functions ($c$-almost periodic functions, semi-$c$-periodic functions) is convolution invariant in order to see that the 
solution $x\mapsto u(x,t_{0}),$ $x\in {\mathbb R}$ is bounded and $c$-uniformly recurrent ($c$-almost periodic, semi-$c$-periodic). 

It is clear that the concepts of Weyl almost periodicity, Doss almost periodicity and Besicovitch-Doss almost periodicity, among many others, can be reconsidered and generalized following the approach based on the use of difference $f(\cdot +\tau)-cf(\cdot)$ in place of the usual one $f(\cdot +\tau)-f(\cdot).$ 
We close the paper with the observation that the class of $c$-almost automorphic functions will be analyzed in our forthcoming paper
\cite{novpinto}.

\end{document}